\documentclass[a4paper]{amsart}
\usepackage[utf8]{inputenc}
\usepackage[english]{babel}
\usepackage{newlfont}
\usepackage{amsmath,amssymb,amsfonts,amsbsy,bbm}
\usepackage[all]{xy}
 \usepackage[cal=euler]{mathalpha}

\usepackage{mathtools,braket,mathrsfs}
\usepackage{tikz}
\usetikzlibrary{arrows}
\usepackage{enumerate}
\usepackage{extarrows}

\newcommand{\bb}{\mathbb}
\newcommand{\mbf}{\mathbf}

\newcommand{\scr}{\mathscr}
\newcommand{\mrm}{\mathrm}
\newcommand*{\bfcdot}{\scalebox{0.6}{$\bullet$}}

\newcommand{\Z}{\ensuremath{\mathbf{Z}}}

\newcommand{\Q}{\ensuremath{\mathbf{Q}}}
\newcommand{\C}{\ensuremath{\mathbf{C}}}
\newcommand{\R}{\ensuremath{\mathbb{R}}}

\newcommand{\p}{\ensuremath{\mathfrak{p}}}
\newcommand{\q}{\ensuremath{\mathfrak{q}}}
\newcommand{\n}{\ensuremath{\mathfrak{n}}}
\newcommand{\m}{\ensuremath{\mathfrak{m}}}

\newcommand{\too}{\longrightarrow}								
\newcommand{\mapstoo}{\longmapsto}
\newcommand{\ontoo}{\relbar\joinrel\twoheadrightarrow}
\newcommand{\intoo}{\lhook\joinrel\too}
\newcommand{\into}{\hookrightarrow}

\newcommand{\PP}{\ensuremath{\mathbb{P}^1}}					
\newcommand{\St}{\textnormal{St}}

\DeclareMathOperator{\Hom}{Hom}

\DeclareMathOperator{\Aut}{Aut}
\DeclareMathOperator{\Gal}{Gal}

\DeclareMathOperator{\GL}{GL}
\DeclareMathOperator{\PGL}{PGL}

\DeclareMathOperator{\red}{red}
\DeclareMathOperator{\pr}{\mathrm{pr}}
\renewcommand{\det}{\operatorname{det}}
\DeclareMathOperator{\ord}{ord}

\newcommand{\plectic}[0]{\text{\textmarried}}

\def\Xint#1{\mathchoice
{\XXint\displaystyle\textstyle{#1}}%
{\XXint\textstyle\scriptstyle{#1}}%
{\XXint\scriptstyle\scriptscriptstyle{#1}}%
{\XXint\scriptscriptstyle\scriptscriptstyle{#1}}%
\!\int}
\def\XXint#1#2#3{{\setbox0=\hbox{$#1{#2#3}{\int}$ }
\vcenter{\hbox{$#2#3$ }}\kern-.585\wd0}}
\def\mint{\Xint\times}

\theoremstyle{definition}
\newtheorem{theorem}{Theorem}[section]
\newtheorem{lemma}[theorem]{Lemma}
\newtheorem{proposition}[theorem]{Proposition}
\newtheorem{corollary}[theorem]{Corollary}
\newtheorem{definition}[theorem]{Definition}

\theoremstyle{remark}
\newtheorem{remark}[theorem]{Remark}
\newtheorem{example}[theorem]{Example}

\begin{document}

\title{Non-Archimedean plectic Jacobians}

\dedicatory{To Massimo Bertolini, on his 60th birthday}

\author{Michele Fornea}
\address{M.~Fornea, Centre de Recerca Matemàtica, Barcelona, Spain.}
\email{mfornea.research@gmail.com}

\author{Lennart Gehrmann}
\address{L.~Gehrmann, Universität Bielefeld, Bielefeld, Germany.}
\email{gehrmann.math@gmail.com}


\begin{abstract}
\emph{Plectic Stark--Heegner points} were recently introduced to explore the arithmetic of higher rank elliptic curves: the concept was inspired by Nekov\'a\v{r} and Scholl's plectic philosophy, while the construction is based on Bertolini and Darmon's groundbreaking use of the $p$-adic uniformization of Shimura curves to study the Birch--Swinnerton-Dyer conjecture.
In this note we give a geometric interpretation of \emph{plectic Heegner points} using the non-Archimedean uniformization of higher-dimensional quaternionic Shimura varieties.
To this end, we define and study a \emph{plectic Jacobian functor} from a category of Mumford varieties to topological groups extending the classical Jacobian functor on Mumford curves.
 
\end{abstract}

\maketitle

\tableofcontents

\section{Introduction}
 Uniformization theorems are used in geometry to organize the diverse array of ``shapes'' found in nature. For instance, Koebe and Poincar\'e's uniformization theorem shows that any Riemann surface is biholomorphic to a quotient of either the sphere, the plane, or the disk by a proper, free and holomorphic action of a discrete group. 
 These kind of classifications are appealing both for their elegance as well as for highlighting a priori hidden structures of geometric objects, making them quite valuable for number theoretic applications.

\smallskip
 The two most iconic uniformization results for the purposes this article are Tate's non-Archimedean uniformization of elliptic curves, whose discovery initiated the field of $p$-adic geometry, and Cerednik--Drinfeld's $p$-adic uniformization of Shimura curves. As we now explain, these results were ingeniously exploited by Bertolini and Darmon to relate Heegner points to special values of $p$-adic $L$-functions or their derivatives (see \cite{MumfordHeegner}, \cite{CDuniformization}, \cite{BDHidaRational}, \cite{RationalitySH}, \cite{BDP}). Cerednik and Drinfeld gave an explicit $p$-adic uniformization of Shimura curves $X_{/\Q}$ of prime-to-$p$ level associated to an indefinite quaternion algebra $B/\Q$ ramified at $p$. They established an isomorphism of rigid analytic varieties
 \[
X(\C_p)\cong\Gamma\backslash\cal{H}_p
 \]
where $\cal{H}_p=\bb{P}^1(\C_p)\setminus\bb{P}^1(\Q_p)$ denotes Drinfeld's upper half plane, and $\Gamma\subset D^\times$ is an explicit $p$-arithmetic subgroup of the definite quaternion algebra $D/\Q$ obtained from $B/\Q$ by switching invariants at $p$ and $\infty$.
In other words, Cerednik and Drinfeld's uniformization theorem identifies the $p$-adic Schottky group associated to $X_{/\Q_p}$, whose existence is ensured by Mumford's general uniformization theorem. Interestingly, that result can be also parleyed into a uniformization of the Abel--Jacobi morphism $\alpha_{X}\colon X\to \mathcal{J}_X$ (see for example \cite{GvdP}, Chapter VI): there exists a lattice $\Lambda\subset\mrm{H}^1(\Gamma,\C_p^\times)$ and an isomorphism of rigid analytic varieties $\mathcal{J}_X(\C_p)\cong\mrm{H}^1(\Gamma,\C_p^\times)/\Lambda$ such that the diagram 
\begin{equation}\label{pAJ}
\xymatrix{
X(\C_p)\ar[r]^-{\alpha_X}\ar@{-}[d]^-\sim& \mathcal{J}_X(\C_p)\ar@{-}[d]^-\sim \\
\Gamma\backslash\cal{H}_p\ar[r]& \mrm{H}^1(\Gamma,\C_p^\times)/\Lambda
}\end{equation}
commutes.
Moreover, the bottom horizontal arrow can be described explicitly in terms of $p$-adic integration. Suppose now that $E_{/\Q}$ is an elliptic curve admitting a non-constant $\Q$-rational morphism $\mathcal{J}_X\to E$ from the Jacobian of $X_{/\Q}$. Tate's uniformization then ensures the existence of a period $q\in p\Z_p\setminus\{0\}$ and an isomorphism of rigid analytic varieties $E(\C_p)\cong \C_p^\times/q^\Z$  fitting in the commuting diagram
\begin{equation}\label{pmodularp}
\xymatrix{
 \mathcal{J}_X(\C_p)\ar@{-}[d]^-\sim\ar[r]& E(\C_p)\ar@{-}[d]^-\sim \\
 \mrm{H}^1(\Gamma,\C_p^\times)/\Lambda\ar[r]&\C_p^\times/q^\Z
}\end{equation}
where the bottom horizontal map arises from a surjection $\mrm{H}^1(\Gamma,\Z)\to\Z$ induced by the modularity of $E_{/\Q}$. The upshot is that by choosing a $p$-adic embedding $\overline{\Q}\hookrightarrow\C_p$ the combination of \eqref{pAJ} and \eqref{pmodularp} provides an explicit rigid analytic description of any modular parametrization $X(\overline{\Q})\to E(\overline{\Q})$. As pointed out previously, this strategy was prominently developed by Bertolini and Darmon to study the BSD conjecture.

\medskip
During the last decade Nekov\'a\v{r} and Scholl have been pursuing a series of \emph{plectic conjectures} supporting the idea that higher dimensional quaternionic Shimura varieties could be used to study the BSD conjecture for higher rank elliptic curves (see \cite{NekHiddenSymmetries}, \cite{NekRubinfest}, \cite{PlecticNS}). The construction of \emph{plectic Stark--Heegner points} is an attempt to provide a theoretical and computational way to test those expectations. Its inception was the realization that the Cerednik--Drinfeld's uniformization of Shimura curves had been generalized  to quaternionic Shimura varieties of arbitrary dimension in \cite{BoutotZink} and \cite[Section 5]{Varshavsky}.
An optimist could then believe that an explicit rigid analytic construction of \emph{global} points on higher rank elliptic curves would be found  by appropriately generalizing the bottom rows of diagrams \eqref{pAJ} and \eqref{pmodularp}. Eventually, plectic Stark--Heegner points were constructed by cohomological means that could deal with arbitrary quadratic extensions of number fields in a uniform fashion. That construction lacked a direct connection to geometry, but it allowed for numerical experiments (\cite{PlecticInvariants}), and to easily relate plectic Stark--Heegner points to higher derivatives of anticyclotomic $p$-adic $L$-functions (\cite{PlecticHeegner}, Theorem A).

\smallskip
 \noindent On the occasion of Bertolini's 60th birthday, we present a construction of \emph{plectic Heegner points} -- i.e., plectic Stark--Heegner points associated to quadratic CM extensions -- from the geometric point of view that first inspired our work (see Section \ref{Sec: Modularity}). We hope it will open new perspectives on the results obtained in \cite{Polyquadratic}, \cite{DarmonFornea} and \cite{IwasawaPlectic}. The technical results of the article focus on the structure of ``plectic subgroups'' of $\PGL_2(F_S)$. In future work we hope to leverage those general considerations to 
\begin{itemize}
\item[$\bfcdot$] construct \emph{plectic Heegner cycles} for higher weight modular forms,
\item[$\bfcdot$] study \emph{plectic comparison theorems} for quaternionic Shimura varieties with Iwahori level at $p$, complementing the main result of \cite{LZplectic}, and
\item[$\bfcdot$] understand the relation between plectic Heegner points over function fields, Heegner--Drinfeld cycles (\cite{ShtukasTaylor}, \cite{ShtukasTaylorII}), and the theory of higher modularity (\cite{HigherMod}).
\end{itemize} 
For the interested reader, we note that there is an Archimedean twin of this paper (\cite{PlecticJacobians}) where an analogue inquiry has been carried out using plectic Hodge theory.

\medskip
\noindent We conclude the introduction with an overview of the main sections. In Sections \ref{Sec: Plectic subgroups} and \ref{Sec: Plectic Jacobians} we work over an arbitrary local field of residue characteristic $p$. Let $\C$ denote the completion of an algebraic closure of either $\Q_p$ or $\mbf{F}_p(\hspace{-0.6mm}(T)\hspace{-0.6mm})$.

Section \ref{Sec: Plectic subgroups} introduces the notion of \emph{plectic groups}.
These are certain discrete subgroups $\Gamma\subseteq\PGL_2(\C)^r$, for some integer $r$, whose  discrete points $\Omega \subseteq \PP(\C)^r$ factors as a product
\[
\Omega=\Omega_1\times\cdots\times\Omega_r
\]
with each $\Omega_i\subseteq \PP(\C)$ an analytic set with compact complement.
Through a reduction to the case $r=1$, the analytic subsets $\Omega_i$'s are shown to be of the same shape as the points of discreteness for the action of $p$-adic Schottky subgroups of $\PGL_2(\C)$.
This permits the definition a building $\mathcal{T}$, which is a product of trees, on which $\Gamma$ acts.
Under the assumption that the quotient $\Gamma\backslash\mathcal{T}$ is a finite complex, in which case $\Gamma$ is called \emph{normal}, several finiteness properties for $\Gamma$ are then deduced. For example, $\Gamma$ is shown to be finitely presented and a virtual duality group.
Moreover, an explicit model for its duality module is given.

Section \ref{Sec: Plectic Jacobians} introduces the notion of \emph{Mumford varieties} whose connected components are quotients of the form $\Gamma\backslash\Omega$ where $\Gamma$ is a torsion-free normal plectic group.
 Plectic Jacobians of Mumford varieties are then defined and shown to be functorial with respect to ``Hecke correspondences'' (see Section \ref{Hecke} for their definition).

Finally, Section \ref{Sec: Modularity} explains how the theory of plectic Jacobians of quaternionic Shimura varieties leads to a generalization of the bottom rows of diagrams \eqref{pAJ} and \eqref{pmodularp}.

\bigskip
\noindent \textbf{Acknowledgements.}
While working on this article the first author was supported by the MTM grant PID2020-118236GB-I00 from Universitat Aut\`onoma de Barcelona, and the Maria de Maeztu grant from the Centre de Recerca Matem\`atica.
The second named author received funding from the Maria Zambrano Grant for the attraction of international talent in Spain and from Deutsche Forschungsgemeinschaft (DFG) via TRR 358 ``Integral Structures in Geometry and Representation Theory''.


\section{Plectic subgroups}\label{Sec: Plectic subgroups}
A \emph{local \'etale algebra} is a finite \'etale algebra over a local field, or in other words,
an algebra of the form
\[F_\Sigma=\prod_{\p\in \Sigma}F_\p\]
with $\Sigma$ is a finite set and the fields $F_\p$ are local fields of the same characteristic and residue characteristic.
We will always assume that the residue characteristic is a fixed prime $p$.
Denote by $\C$ the completion of a fixed algebraic closure of $\Q_p$ (resp.~$\mbf{F}_p(\hspace{-0.6mm}(T)\hspace{-0.6mm})$) in case of characteristic $0$ (resp.~characteristic $p$).

\begin{definition}
An \emph{embedding} of a local \'etale algebra $F_\Sigma$ is a tuple $\iota=(\iota_\p)_{\p\in \Sigma}$ of ring homomorphisms $\iota_\p \colon F_\p \into \C$.
We call an \emph{embedded local \'etale algebra} any local \'etale algebra with the data of an embedding.
\end{definition}

Let be $(F_\Sigma,\iota)$ an embedded local \'etale algebra.
For any $\p\in\Sigma$ the valuation ring of $F_\p$ is denoted by $\mathcal{O}_\p$ with maximal ideal $\p$.
Moreover, we put
\[K_\p=\PGL_2(\mathcal{O}_\p),\]
and write
\[\mrm{Iw}_\p\subseteq \PGL_2(\mathcal{O}_\p)\]
for the Iwahori subgroup of all integral matrices which are congruent to an upper triangular matrix modulo $\p$.
For any subset $S \subseteq \Sigma$ we set
\[K_{S}=\prod_{\p\in S}K_\p\quad \mbox{and}\quad\mrm{Iw}_{S}=\prod_{\p\in S}\mrm{Iw}_\p.\]
Given a non-empty subset $S\subseteq \Sigma$ and a subfield $E\subseteq \C$ containing $\iota_\p(F_\p)$ for all $\p\in S$ we put
\[F_{S}=\prod_{\p\in S} F_\p \quad \mbox{and}\quad E_{S}=\prod_{\p\in S} E.\]
We usually abbreviate $E_\p:=E_{\{\p\}}$.
The embedding $\iota$ induces a continuous action of the group $\PGL_2(F_{S})=\prod_{\p\in S} \PGL_2(F_\p)$ on the product
\[\PP(E_{S})=\prod_{\p \in S} \PP(E)\]
of projective lines.
Given subsets $S'\subseteq S\subseteq \Sigma$ we often view $\PGL_2(F_{S'})$ as a subgroup of $\PGL_2(F_{S})$ via the canonical embedding, and we write
\[
\pr_{S'}\colon\PGL_2(F_{S}) \too \PGL_2(F_{S'})
\]
for the canonical projection.

\subsection{Limit points}
For the remainder of this chapter we fix an embedded local \'etale algebra algebra $(F_\Sigma,\iota)$. 
We collect basic facts about the set of limit points of the action of subgroups of $\PGL_2(F_S)$ on $\PP(\C_S)$ for subsets $S\subseteq \Sigma$.
 \begin{definition}
	For a subgroup $\Gamma$ of $\PGL_2(F_{S})$, we define $\mathcal{L}^{S}_{\Gamma}$ to be the set of limit points of $\Gamma$ in $\PP(\C_{S})$, i.e., the set of elements $x\in\PP(\C_{S})$ such that there exists $y\in\PP(\C_{S})$ and distinct elements $\{\gamma_j\}_{j\ge1}\subseteq\Gamma$ such that $\gamma_j(y)$ converges to $x$.
 \end{definition}
\noindent We may consider a subgroup $\Gamma\subseteq \PGL_2(F_{S})$ also as a subgroup of $\PGL_2(F_{\Sigma})$ via the canonical embedding.
It follows directly from the definition that
\begin{align}\label{changeofset}
\mathcal{L}_{\Gamma}^\Sigma=\mathcal{L}^{S}_{\Gamma}\times \PP(\C_{\Sigma\setminus S}).
\end{align}
The compactness of the projective line implies that the set of limit points is rarely empty:
\begin{proposition}\label{nolimits}
Let $\Gamma\subseteq \PGL_2(F_S)$ a subgroup with $\mathcal{L}_\Gamma^S=\emptyset$.
Then $\Gamma$ is finite.
\end{proposition}
\begin{proof}
Let $E\subseteq \C$ be a subfield finite over $F$ that contains the compositum of the fields $\iota_\p(F_\p)$ for $\p\in S$.
By assumption, for any $x\in \PP(E_S)$ the orbit $\Gamma.x$  is a discrete subset of the compact space $\PP(E_S)$. Thus, it is finite.
Choosing $E$ large enough there exists a point $x\in \PP(E_S)$ whose stabilizer in $\PGL_2(F_S)$ is finite, hence $\Gamma$ is finite as well.
\end{proof}

Commensurable group have the same set of limit points:
\begin{lemma}\label{finiteindex}
Let $\Gamma,\Gamma'\subseteq \PGL_2(F_{S})$ be commensurable subgroups, then
\[
\mathcal{L}_{\Gamma'}^S=\mathcal{L}_\Gamma^S.
\]	
\end{lemma}
\begin{proof}
It is enough to assume that $\Gamma'$ is a finite index subgroup of $\Gamma$.
	Clearly $\mathcal{L}_{\Gamma'}^S\subseteq\mathcal{L}_\Gamma^S$.
	 Let now $x$ be an element of $\mathcal{L}_\Gamma^S$.
	By definition, there exists $y\in\PP(\C_{S})$ and distinct elements $\{\gamma_j\}_j\subseteq\Gamma$ such that $\gamma_j(y)$ converges to $x$.
	Since there are finitely many cosets $\Gamma'\backslash\Gamma$, up to considering a subsequence, there exists a fixed $g\in\Gamma$ such that for every index 
	\[
	\gamma_j=\delta_j\cdot g\quad\mbox{for some}\quad \delta_j\in\Gamma'.
	\]
	We put $y_0=g(y)$.
	Then by construction the sequence $\delta_j(y_0)$ converges to $x$.
\end{proof}

If $\Gamma \subseteq \PGL_2(F_{S})$ is not discrete, we clearly have that $\mathcal{L}_{\Gamma}^S= \PP(\C_{S})$. Therefore, the set of limit points is most interesting for discrete subgroups.
Let us assume momentarily that $\Gamma$ is a discrete subgroup of $\PGL_2(F_\p)$ for some $\p\in \Sigma$.
From (1.6.4) of \cite[Section I]{GvdP}, we know that $\mathcal{L}_{\Gamma}^{\{\p\}}$ is a subset of $\PP(F_\p)$.
A mild generalization of the proof of \textit{loc.cit.}~provides the following rough description of the shape of limits sets of general discrete subgroups of $\PGL_2(F_{S})$.
\begin{lemma}\label{discrete}
Let  $\Gamma \subseteq \PGL_2(F_{S})$ be a discrete subgroup, then
$$\mathcal{L}_\Gamma^S\subseteq \bigcup_{\p\in S} \Big(\PP(F_\p) \times \PP(\C_{S\setminus \{\p\}})\Big).$$
\end{lemma}
\begin{proof}
Let $x$ be an element of $\mathcal{L}_\Gamma^S$.
We have to show that there exists $\p\in S$ such that the $\p$-th component of $x$ is $F_\p$-rational.
From the definition of limit set, there exists a point $y\in \PP(\C_{S})$ and a sequence of distinct elements $\{\gamma_j\}_j\subseteq \Gamma$ such that $\gamma_j(y)$ converges to $x$.
At the cost of replacing $\{\gamma_j\}_j$ with a subsequence, we may decompose $S=S_1\cup S_2$ such that the sequence of elements $(\gamma_j)_\p$ consists of distinct elements for all $\p \in S_1$ and are constant for all $\p\in S_2$.
By assumption the subset $S_1$ is not empty.
Then, as in the proof of (1.6.4) of \cite[Section I]{GvdP},  we can find preimages $\widetilde{\gamma_j}$ of $\gamma_j \in \GL_2(F_S)$ such that
\begin{itemize}
\item[$\bfcdot$] the sequence $(\widetilde{\gamma_j})_\p$ is constant for all $\p \in S_2$, and
\item[$\bfcdot$] for every $\p\in S_1$ the sequence $(\widetilde{\gamma}_j)_\p$ converges in $\mrm{M}_2(F_\p)$ (after a change of variables) to a matrix of the form
$$\begin{pmatrix}
a_\p & b_\p \\
1 & d_\p
\end{pmatrix}.$$
\end{itemize}
Since the subgroup $\Gamma$ is discrete, the limit of the converging sequence $\{\widetilde{\gamma}_j\}_j$ is not an element of $\GL_2(F_S)$, i.e., there exists a $\p \in S_1$ such that $a_\p d_\p = b_\p$.
In particular, the limit of the sequence $(\gamma_j(y))_\p$ is either $a_\p$ or $\infty$, which proves the claim.
\end{proof}

\subsection{Plectic groups}
Consider a discrete subgroup of $\PGL_2(F_{S})$ of the form
\[
\Gamma=\prod_{\p \in S} \Gamma_\p,
\]
where $\Gamma_\p\subseteq \PGL_2(F_\p)$ is a discrete subgroup for each $\p\in S$.
An easy calculation shows that the set of limit points of $\Gamma$ is given by
\[
\mathcal{L}^{S}_{\Gamma} = \bigcup_{\p \in S} \Big(\mathcal{L}_{\Gamma_\p}^{\{\p\}}\times\PP(\C_{S\setminus\{\p\}})\Big).
\]
We are interested in studying the properties of all discrete subgroups of $\PGL_2(F_{S})$ whose set of limit points has the same structure.

\begin{definition}
A subgroup $\Gamma\subseteq \PGL_2(F_{S})$ is called \emph{plectic}, if there exist subsets $\mathcal{L}_{\Gamma,\p}\subseteq\PP(F_\mathfrak{p})$ for $\p\in S$ such that
		\[
		\mathcal{L}_{\Gamma}^{S}=\bigcup_{\p\in S} \Big(\mathcal{L}_{\Gamma,\p}\times \PP(\C_{S\setminus \{\p\}})\Big).
		\]
		In that case, we define the subset of \emph{distinguished limit points} as
		$$\mathcal{L}_{\Gamma,S}:=\prod_{\p\in S} \mathcal{L}_{\Gamma,\p}.$$
\end{definition}
\begin{remark}
    The set of distinguished limit points can be thought of as a $p$-adic analogue of the notion of Bergman--Shilov boundary in the theory of holomorphic functions in several complex variables (cf.~\cite{Fuks}, Ch.~III--15). 
\end{remark}
Equation \eqref{changeofset} implies that a subgroup $\Gamma\subseteq \PGL_2(F_{S})$ is plectic if and only if it is plectic (with $\mathcal{L}_{\Gamma,\p}=\emptyset$ for all $\p\in \Sigma\setminus S$) when viewed as a subgroup of $\PGL_2(F_\Sigma)$.
By Lemma \ref{finiteindex} every subgroup commensurable with a plectic subgroup is plectic.
As noted earlier, every plectic subgroup is automatically discrete.
The following example shows that the converse does not hold in general.

\begin{example}
The subgroup $\Gamma\subseteq \PGL_2(\Q_p)\times \PGL_2(\Q_p)$ generated by the tuple
$$\left(\begin{pmatrix}p & 0 \\ 0 & 1\end{pmatrix},\begin{pmatrix}p & 0 \\ 0 & 1\end{pmatrix}\right).$$
is discrete but not plectic.
Indeed, its set of limit points in $\PP(\C)\times \PP(\C)$ is
$$\mathcal{L}_\Gamma=\{(0,0),(\infty,\infty)\}.$$
\end{example}

Evidently, the easiest way of generating plectic subgroups is taking products of plectic subgroups:
\begin{lemma}\label{product}
Let $S=S_1\cup S_2$ with $S_1\cap S_2=\emptyset$.
The product $\Gamma=\Gamma_{1}\times \Gamma_{2}$ of two plectic subgroups $\Gamma_{i}\subseteq \PGL_2(F_{S_i})$, $i=1,2$, is a plectic subgroup of $\PGL_2(F_{S})$ with set of distinguished limit points given by
$$\mathcal{L}_{\Gamma,S}=\mathcal{L}_{\Gamma_{1,S_1}}\times\mathcal{L}_{\Gamma_{2,S_2}}.$$
\end{lemma}

A more interesting collection of plectic subgroups is given by cocompact lattices in $\PGL_2(F_S)$:
\begin{proposition}\label{cocompact}
Let $\Gamma\subseteq \PGL_2(F_{S})$ be a discrete and cocompact subgroup.
Then, $\Gamma$ is plectic with set of distinguished limit points
$$\mathcal{L}_{\Gamma,S}=\PP(F_S).$$
\end{proposition}
\begin{proof}
 By Lemma \ref{discrete} it is enough to show that
$$\bigcup_{\p\in S} \Big(\PP(F_\p) \times \PP(\C_{S\setminus \{\p\}})\Big)\subseteq \mathcal{L}_\Gamma^{S}.$$
 So let us fix an element $\q\in S$ and a point $x=(x_\p)_{\p\in S}\in \PP(\C_{S})$ with $x_\q\in \PP(F_\q)$.
 We consider the elements $g_j \in\mrm{GL}_2(F_\q) \subseteq \mrm{GL}_2(F_{S})$, $j\geq 1$, given by
	\[
	g_j= \begin{pmatrix}
			p^{j}&0\\
		1&1
		\end{pmatrix}\quad\mbox{if}\ x_\q=0,\qquad
	g_j= \begin{pmatrix}
		1&0\\
		x_\q^{-1}&p^{j}
	\end{pmatrix}\quad\mbox{if}\ x_\q\not=0.
	\]
	Let $y=(y_\p)_{\p\in S}\in \PP(\C_{S})$ be the point given by $y_\q=\infty$ and $y_\p=x_\p$ for $\p\neq \q$.
	Then, it is easy to see that
	\[
	\underset{j\to\infty}{\lim}\ g_j(y)=x.
	\]
	Since $\Gamma$ is cocompact, the double coset $\Gamma\backslash\PGL_2(F_{S})/\mrm{Iw}_{S}$ is finite.
	Therefore, we can suppose, up to replacing $\{g_j\}_{j}$ with a subsequence, that there is a fixed $g\in \mrm{GL}_2(F_{S})$ such that 
	\[
	g_j=\gamma_j\cdot g\cdot r_j\qquad \mbox{with}\ \gamma_j\in \Gamma,\ r_j\in \mrm{Iw}_{S}.
	\]
	Note that the sequence $\{\gamma_j\}_j$ consist of pairwise different elements because $g_{j}^{-1}\cdot g_k\not\in \mrm{Iw}_{S}$ when $j\not=k$.
	As the group $\mrm{Iw_{S}}$ is compact, so is the orbit of $y$ under $\mrm{Iw}_{S}$.
	In particular, there exists a subsequence of $\{r_j\}_{j\geq 1}$ such that $(g\cdot r_j)(y)$ converges to some $y_\circ \in \PP(\C_{S})$.
We deduce the existence of a subsequence $\{\gamma_\ell\}_\ell\subseteq\{\gamma_j\}_j$ with
	\[
	x=\underset{\ell\to\infty}{\lim}\ \gamma_{\ell}(y_\circ),
	\] 
	that is, $x$ is a limit point for $\Gamma$.
\end{proof}
Note that for arithmetic application one can produce plectic groups that are not products by considering $S$-arithmetic subgroups of totally definite quaternion algebras (see Section \ref{Hilbert}). The following lemma will be used to construct a Hecke action on plectic Jacobians.
\begin{lemma}\label{action}
Let $\Gamma\subseteq \PGL_2(F_S)$ be a plectic subgroup and $g$ an element of $\PGL_2(F_S)$.
Then $g\Gamma g^{-1}$ is plectic with set of distinguished limit points
$$\mathcal{L}_{g\Gamma g^{-1},S}=g( \mathcal{L}_{\Gamma,S}).$$
\end{lemma}

\subsection{Partial intersection with compact open subgroups}
In case $S$ is a singleton $\{\p\}$, the set of limits points $\mathcal{L}_{\Gamma,\p}\subseteq \PP(F_\p)$ of a discrete subgroup $\Gamma\subseteq \PGL_2(F_\p)$ is either a finite set with at most $2$ elements or a perfect set, that is, a closed subset without isolated points (see (1.6.3) of \cite[Chapter I]{GvdP}).
We now prove the analogous statement for plectic subgroups for arbitrary $S$ by reducing to the case of a singleton set. 

\smallskip
\noindent 
Suppose $S_1\subseteq S$. Given subgroups $\Gamma\subseteq \PGL_2(F_{S})$ and $U\subseteq \PGL_2(F_{S\setminus S_1})$, we put
\begin{align*}
\widetilde{\Gamma}_U&:=\Gamma \cap \left(\PGL_2(F_{S_1})\times U\right)\\
\intertext{and}
\Gamma_U&:=\pr_{S_1}(\widetilde{\Gamma}_U)\subseteq \PGL_2(F_{S_1}).
\end{align*}

\begin{lemma}\label{finitekernel}
Let $\Gamma\subseteq \PGL_2(F_{S})$ be a discrete subgroup and $U\subseteq \PGL_2(F_{S\setminus S_1})$ an open compact subgroup.
Then $\Gamma_U\subseteq \PGL_2(F_{S_1})$ is a discrete subgroup, and the kernel of the projection map $\pr_{S_1}\colon \widetilde{\Gamma}_U \to \Gamma_U$ is finite. 
\end{lemma}
\begin{proof}
Let $V\subseteq \PGL_2(F_{S_1})$ be any compact open subgroup.
Since $\Gamma$ is discrete, the intersection $\Gamma\cap(V\times U)$ is finite.
This implies that the intersection $\Gamma_U\cap V$ is finite, which proves the first claim.
The kernel of the projection from $\widetilde{\Gamma}_U$ to $\Gamma_U$ is given by the intersection
\[
\Gamma \cap \big(\{1_{\PGL_2(F_{S_1})}\}\times U\big),
\]
which is discrete as well as compact. Hence, it is finite.
\end{proof}

\begin{proposition}\label{induction}
Let $\Gamma\subseteq \PGL_2(F_{S})$ be a plectic subgroup.
For every open compact subgroup $U\subseteq \PGL_2(F_{S\setminus S_1})$ the subgroup $\Gamma_U\subseteq \PGL_2(F_{S_1})$ is plectic with
\[
\mathcal{L}_{\Gamma,\p}=\mathcal{L}_{\Gamma_U,\p}
\]
for all $\p\in S_1$.
\end{proposition}
\begin{proof}
Let $U_\q\subseteq \PGL_2(F_\p)$, $\q\in S\setminus S_1$, be open compact subgroups with
\[
U^\prime:=\prod_{\q\in S\setminus S_1}U_\q\ \subseteq\ U.
\]
By Lemma \ref{finiteindex} one may replace $U$ by $U^\prime$.
Thus, by an easy inductive argument we reduce to the case that $S\setminus S_1=\{\q\}$ is a singleton set.
Moreover, by invoking Lemma \ref{finiteindex} again we may assume that $U=K_\p$.

\smallskip
\noindent Now, let $x$ be an element of $\mathcal{L}_{\Gamma_{U}}^{S_1}$.
Choose $y\in\PP(\C_{S_1})$ and distinct elements $\{\gamma_j\}_j\subseteq \Gamma_U$ such that $\gamma_j(y)$ converges to $y$.
Moreover, pick preimages $\widetilde{\gamma}_j$ of $\widetilde{\gamma}_j$ in $\widetilde{\Gamma}_U$ and an element $y_\q \in \PP(\C_\q)\setminus \PP(F_\q)$.
The sequence $\widetilde{\gamma}_j(y_\q)$ lies inside the compact set $U.x_\q\subseteq \PP(\C_\q)\setminus \PP(F_\q)$.
Since $\PP(\C_\q)\setminus \PP(F_\q)$ is a metric space, it follows that $U.x_\q$ is sequentially compact.
In particular, we may replace $\widetilde{\gamma}_j$ by a subsequence such that $\widetilde{\gamma}_j(y_\q)$ converges to an element $x_\q \in \PP(\C_\q)\setminus \PP(F_\q)$.
It follows that $(x,x_\q)$ is an element of $\mathcal{L}_{\Gamma}^{S}$.
Since $\Gamma$ is plectic and $x_\q$ is not an element of $\mathcal{L}_{\Gamma,\p}$, it follows that $(x,x_\q)\in \mathcal{L}_{\Gamma}^{S}$ for all $x_\q\in \PP(\C_\q)$.

\smallskip
\noindent On the other hand, let $x\in \PP(\C_{S_1})$ be an element such that there exists $x_\q\in \PP(\C_\q) \setminus \mathcal{L}_{\Gamma,\p}$ with $(x,x_\q)\in \mathcal{L}_{\Gamma}^{S}$. Then $(x,x_\q)$ lies in  $\mathcal{L}_{\Gamma}^{S}$ for all $x_\q\in \PP(\C_\q)$.
Consider the reduction map
\[
\red\colon \PP(\C_\q)\setminus \PP(F_\q) \too \mathcal{T}_\q
\]
to the Bruhat--Tits tree of $\PGL_2(F_\q)$ (see for example \cite[Section 4.9]{FvdP}).
The subgroup $U=K_\p\subseteq \PGL_2(F_\q)$ is the stabilizer of a vertex $v \in \mathcal{T}_\q$. 
Fix elements $x_\q \in \red^{-1}(v)$, $y=(y_\p)\in\PP(\C_{S_1})$, $y_\q\in \PP(\C_{S_1})$, and a sequence of distinct elements $\{\gamma_j\}_j$ satisfying 
\[
\lim_{j\to \infty} \gamma_j(y,y_\q)=(x,x_\q).
\]
Since $\red^{-1}(v)$ is an open neighbourhood of $x_\q$, we may assume that $\gamma_j(y_\q)\in\red^{-1}(v)$ for all $j\geq 0$.
This implies that $\gamma^\prime_j:=\gamma_j \gamma_0^{-1}\in \Gamma_U$ for all $j$.
By Lemma \ref{finitekernel} we may choose a subsequence of the sequence $\pr_{S_1}(\gamma^\prime_j)$ that consists of distinct elements.
It follows that $x$ is an element of $\mathcal{L}_{\Gamma_U}^{S_1}$.
\end{proof}

The following characterization of limit sets of plectic subgroups can be deduced by combining Proposition \ref{induction} with $S_1=\{\p\}$ and equations (1.5) and (1.6.3) of \cite[Chapter I]{GvdP}.
\begin{corollary}\label{limitsetcases}
Let $\Gamma\subseteq \PGL_2(F_{S})$ be a plectic subgroup.
For every $\p\in S$ one of the following is true:
\begin{enumerate}[(i)]
\item\label{case0} $\mathcal{L}_{\Gamma,\p}=\emptyset$.
\item\label{caseno} $|\mathcal{L}_{\Gamma,\p}|=1$ and $\mathrm{char}(F)=p$.
\item\label{case1} $|\mathcal{L}_{\Gamma,\p}|=2$.
\item\label{case2} $\mathcal{L}_{\Gamma,\p}$ is a perfect set.
\end{enumerate} 
\end{corollary}

\noindent Case \eqref{caseno} only happens in the presence of a large subgroup consisting of unipotent elements, while case \eqref{case0} and \eqref{caseno} imply the absence of hyperbolic elements  (see Examples (1.7) of \cite[Chapter I]{GvdP}). This motivates the following definition:

\begin{definition}
A plectic subgroup $\Gamma\subseteq \PGL_2(F_{S})$ is called \emph{semi-simple} if for every $\p\in S$ the set $\mathcal{L}_{\Gamma,\p}$ does not have cardinality $1$.
Given a semi-simple plectic subgroup $\Gamma$ we put
\[
\mbf{S}_{\Gamma}:=\left\{\p\in S\ \vert\ \mathcal{L}_{\Gamma,\p}\neq \emptyset \right\}.
\]
\end{definition}

Clearly, every subgroup commensurable to a semi-simple plectic subgroup is semi-simple.
Moreover, when $S_1, S_2\subseteq S$ are disjoint subsets and $\Gamma_i\subseteq \PGL_2(F_{S_i})$ are semi-simple plectic subgroups, then also the product $\Gamma_1 \times \Gamma_2$ is plectic semi-simple.
Proposition \ref{induction} immediately implies the following:
\begin{corollary}
Let $\Gamma\subseteq \PGL_2(F_{S})$ be a semi-simple plectic subgroup and $U\subseteq \PGL_2(F_{S\setminus S_1})$ an open compact subgroup.
Then $\Gamma_U$ is a semi-simple plectic subgroup.
\end{corollary}

\begin{lemma}
Let $\Gamma\subseteq \PGL_2(F_S)$ be a torsion-free plectic subgroup.
Then $\Gamma$ is semi-simple.
\end{lemma}
\begin{proof}
This is a consequence of \cite[Ch. I, (1.7)]{GvdP} and the fact that unipotent elements in characteristic $p$ are torsion.
\end{proof}

\subsection{Associated polysimplicial complex}
Recall that an infinite halfline of a tree is called an end of that tree.
The set of ends of the Bruhat--Tits tree $\mathcal{T}_\p$, $\p\in S$, can canonically be identified with $\PP(F_\p)$ (\cite{DasTei}, Lemma 1.3.6).
Let $\Gamma\subseteq \PGL_2(F_{S})$ be a semi-simple plectic subgroup.
For every $\p \in S$ one associates a locally finite tree $\mathcal{T}_{\mathcal{L}_{\Gamma,\p}}$ to the set of limit points $\mathcal{L}_{\Gamma,\p}$:
\begin{itemize}
    \item [$\bfcdot$] When $\mathcal{L}_{\Gamma,\p}=\emptyset$, the tree $\mathcal{T}_{\mathcal{L}_{\Gamma,\p}}$ consists of a single vertex.
 \item [$\bfcdot$] If $\mathcal{L}_{\Gamma,\p}$ is a perfect set, the construction of $\mathcal{T}_{\mathcal{L}_{\Gamma,\p}}$ is given in \cite[Section I.2]{GvdP}, and \cite[Section 4.9]{FvdP}. In a nutshell, it is smallest subtree of the Bruhat--Tits tree $\mathcal{T}_\p$ whose set of ends is $\mathcal{L}_{\Gamma,\p}$. It follows directly from the construction that $\pr_{\p}(\Gamma)$ acts on $\mathcal{T}_{\mathcal{L}_{\Gamma,\p}}$. 
 \item [$\bfcdot$] In case $\lvert\mathcal{L}_{\Gamma,\p}\rvert=2$, the pointwise stabilizer of $\mathcal{L}_{\Gamma,\p}$ in $\PGL_2(F_\p)$ is a maximal $F_\p$-split torus $T$ and $\pr_{\p}(\Gamma)$ is contained in the normalizer $N(T)$ of $T$.
Let $K_T\subseteq T$ the maximal compact subgroup. The quotient $T/K_T$ is an infinite cyclic group.
One defines $\mathcal{T}_{\mathcal{L}_{\Gamma,\p}}$ as the Cayley graph associated to a generator of $T/K_T$.
The normalizer $N(T)$ and, thus, also $\pr_\p(\Gamma)$ clearly act on $\mathcal{T}_{\mathcal{L}_{\Gamma,\p}}$.
It is possible to identify $\mathcal{T}_{\mathcal{L}_{\Gamma,\p}}$ with a subtree of the Bruhat--Tits tree $\mathcal{T}_\p$ of $\PGL_2(F_\p)$. More precisely, $\mathcal{T}_{\mathcal{L}_{\Gamma,\p}}$ is identified with the infinite geodesic connecting the two points in $\mathcal{L}_{\Gamma,\p}\subseteq \PP(F_\p)$.
\end{itemize}
Combining the action for all $\p\in S$ we get an action of $\Gamma$ on the locally finite polysimplical complex
\[
\mathcal{T}_\Gamma:=\prod_{\p\in S} \mathcal{T}_{\mathcal{L}_{\Gamma,\p}}.
\]
The dimension of $\mathcal{T}_\Gamma$ is given by
\[
\dim(\mathcal{T}_\Gamma)=\left|\mbf{S}_{\Gamma}\right|.
\]

\begin{lemma}\label{finitesimplex}
Let $\Gamma\subseteq \PGL_2(F_{S})$ be semi-simple plectic subgroup.
The action of $\Gamma$ on $\mathcal{T}_\Gamma$ is proper:
the stabilizer of each simplex $\sigma$ of $\mathcal{T}_\Gamma$ in $\Gamma$ is finite.
\end{lemma}
\begin{proof}
For every $\p \in \mbf{S}_\Gamma$ the tree $ \mathcal{T}_{\mathcal{L}_{\Gamma,\p}}$ is a subtree of the Bruhat--Tits tree $\mathcal{T}_\p$.
The stabilizer in $\PGL_2(F_\p)$ of a finite subgraph of the Bruhat--Tits tree $\mathcal{T}_\p$ is a compact open subgroup.
Thus, for any simplex $\sigma$ of $\mathcal{T}_\Gamma$ there exists exists a compact open subgroup $U\subseteq \PGL_2(F_{\mbf{S}_\Gamma})$ such that
\[
\mathrm{Stab}_\Gamma(\sigma)= \widetilde{\Gamma}_U.
\]
By Lemma \ref{induction} the group $\Gamma_U$ does not have any limit points and, therefore, it is finite by Proposition \ref{nolimits}.
Lemma \ref{finitekernel} implies that $\widetilde{\Gamma}_U$ is finite, as well.
\end{proof}

\noindent The converse to Lemma \ref{finitesimplex} also holds.
Indeed, by \cite[Proposition 3.2]{TitsTrees}, every finite order automorphism of a tree fixes a vertex or an edge.
Therefore, we get:
\begin{lemma}\label{finiteorder}
Let $\Gamma\subseteq \PGL_2(F_{S})$ be semi-simple plectic subgroup and $\gamma\in\Gamma$ an element of finite order.
Then $\gamma$ fixes a simplex of $\mathcal{T}_\Gamma$.
\end{lemma}

\noindent In accordance with \cite[Definition, p.~192]{Mustafin}, we consider the following finiteness condition:
\begin{definition}
A semi-simple plectic subgroup $\Gamma\subseteq \PGL_2(F_{S})$ is called a \emph{normal plectic subgroup} if the quotient
$\Gamma\backslash \mathcal{T}_\Gamma$ is a finite polysimplical complex.
\end{definition}
\noindent Every subgroup commensurable to a normal plectic subgroup is normal.
Moreover, if $S_1, S_2 \subseteq S$ are disjoint subsets and $\Gamma_i\subseteq \PGL_2(F_{S_i})$ are normal plectic subgroup, then the product $\Gamma_1\times \Gamma_2$ is normal.

\begin{example}
Any discrete cocompact subgroup $\Gamma\subseteq \PGL_2(F_{S})$ is a normal plectic subgroup.
\end{example}

\noindent A semi-simple subgroup $\Gamma\subseteq \PGL_2(F_\p)$ is normal if and only if it is finitely generated (\cite{GvdP}, Ch. I, Lemma 3.2.2).
In higher dimension one implication still holds.
Indeed, the hypothesis of \cite[Theorem 4]{Brownpresentations} are fulfilled by Lemma \ref{finitesimplex}. Therefore: 
\begin{proposition}
Every normal plectic subgroup is finitely presented.
\end{proposition}

Normal plectic subgroups also have nice cohomological finiteness properties as the next two propositions show.
For cohomological notions such as groups of type (WFL) and duality groups we refer to \cite[Section VIII]{Brownbook}.
\begin{proposition}\label{WFL}
Every normal plectic subgroup is of type (WFL).
\end{proposition}
\begin{proof}
By \cite[Theorem 11.1]{Brownbook} it is enough to show that every normal plectic subgroup $\Gamma\subseteq \PGL_2(F_{S})$ is virtually torsion-free.
This can be argued as in the proof of \cite[Ch. I.3, Theorem 3.1]{GvdP}:
let $\{\sigma_1,\ldots,\sigma_r\}$ be a set of simplices of $\mathcal{T}_\Gamma$ for the orbits of the $\Gamma$-action.
By Lemma \ref{finiteorder}, every element of finite order of $\Gamma$ is conjugated to an element that fixes some $\sigma_i$, $1\leq i\leq r$.
Thus, Lemma \ref{finitesimplex} implies that the set of conjugacy classes of finite order elements of $\Gamma$ is finite.
Take a set of representatives $\mathcal{R}:=\{\gamma_1,\ldots,\gamma_t\}$ of these conjugacy classes.
Since $\Gamma$ is finitely generated, there exist subrings $R_\p\subseteq F_\p$, $\p\in S$, which are finitely generated over their prime ring and such that $\Gamma\subseteq \prod_{\p\in S}\PGL_2(R_\p)$.
Choose maximal ideals $\m_\p\subseteq R_\p$ such that the intersection of $\mathcal{R}$ with the kernel of the map
\begin{equation}\label{mapp}
\Gamma \intoo \prod_{\p\in S}\PGL_2(R_\p) \ontoo \prod_{\p\in S}\PGL_2(R_\p/\m_\p)
\end{equation}
consists only of the unit element.
The kernel of \eqref{mapp} is a torsion-free normal subgroup of $\Gamma$ of finite index.
\end{proof}

Before stating the next proposition we need to introduce some notations.
\begin{definition}
Let $\Gamma\subseteq \PGL_2(F_S)$ a plectic subgroup and $\p$ an element of $S$.
The \emph{$\p$-th Steinberg module} attached to $\Gamma$ is the free abelian group
\[
\St_{\Gamma,\p}=C(\mathcal{L}_{\Gamma,\p},\Z)/\Z
\]
of locally constant $\Z$-valued functions on $\mathcal{L}_{\Gamma,\p}$ modulo constant functions.
It naturally carries a $\Gamma$-action.
The \emph{$S$-Steinberg module} of $\Gamma$ is defined as the tensor product
\[
\St_{\Gamma,S}=\bigotimes_{\p\in S} \St_{\Gamma,\p}
\]
of abelian groups.
The \emph{$S$-orientation character} of $\PGL(F_{S})$ is the homomorphism
\[
\chi_{S}^{\mathrm{or}}\colon \PGL(F_{S})\too \{\pm 1\},\quad (g_\p)_{\p\in S}\mapstoo (-1)^{\sum_{\p\in S} \ord_\p(\det(g_\p))}.
\]
\end{definition}
Note that if $\mathcal{L}_{\Gamma,\p}=\PP(F_\p)$, the module $\St_{\Gamma,\p}$ is the smooth $\Z$-valued Steinberg representation of $\PGL_2(F_\p)$; while, if $\mathcal{L}_{\Gamma,\p}$ is empty, then $\St_{\p}(\Gamma)=\{0\}$.
This implies that $\St_{\Gamma,S}=0$ unless $S=\mbf{S}_{\Gamma}$.
As a product of trees, the polysimplical complex $\mathcal{T}_\Gamma$ is orientable.
Furthermore, the action of an element $\gamma\in\Gamma$ preserves the orientation on $\mathcal{T}_\Gamma$ if and only if $\chi_{\mbf{S}_{\Gamma}}^{\mathrm{or}}(\gamma)=1$.

\begin{proposition}\label{duality}
Let $\Gamma$ be a normal plectic subgroup. Then $\Gamma$ is a virtual duality group of virtual cohomological dimension
\[
\mathrm{vcd}(\Gamma)=\dim(\mathcal{T}_\Gamma).
\]
Its duality module is equal to $\big(\St_{\Gamma,\mbf{S}_{\Gamma}}\big)\otimes\chi_{\mbf{S}_{\Gamma}}^{\mathrm{or}}$.
\end{proposition}
\begin{proof}
By Proposition 7.5 and Theorem 10.1 of \cite[Chapter VIII]{Brownbook} it is enough to prove that the integral cohomology with compact supports of $\mathcal{T}_\Gamma$ vanishes outside of degree $d=\dim(\mathcal{T}_\Gamma)$ and is equal to $\St_{\Gamma,\mbf{S}_{\Gamma}}$ in degree $d$.
For $\p\in \mbf{S}_{\Gamma}$ the local tree $\mathcal{T}_{\Gamma,\p}$ has no leaves and its set of endpoints is exactly $\mathcal{L}_{\Gamma,\p}$.
A direct calculation shows that the cohomology with compact supports of $\mathcal{T}_{\Gamma,\p}$ is supported in degree $1$ and equal to $\St_{\Gamma,\p}$.
Thus, the claim follows from the Künneth formula for cohomology with compact supports.
\end{proof}

\begin{corollary}\label{fingen}
Let $\Gamma$ be a normal plectic subgroup. 
Then the module of $\Gamma$-coinvariants of $\St_{\Gamma,S}$ is a finitely generated abelian group.
\end{corollary}
\begin{proof}
If $S\neq \mbf{S}_{\Gamma}$, then $\St_{\Gamma,S}=0$ and the claim is obvious.
So suppose that $S=\mbf{S}_{\Gamma}$.
In case $\Gamma$ is torsion-free, Proposition \ref{duality} implies that
\[
\left(\St_{\Gamma,S}\right)_{\Gamma}:=\mrm{H}_0\big(\Gamma,\St_{\Gamma,S}\big)\cong\mrm{H}^{|S|}\big(\Gamma, \chi_{S}^\mrm{or}\big)
\]
The right hand side is finitely generated since $\Gamma$ is of type $(FL)$ by Proposition \ref{WFL}.
In case $\Gamma$ is not torsion-free, the claim follows by passing to a torsion-free finite-index subgroup and using the Hochschild--Serre spectral sequence.
\end{proof}

We end this section by showing that the property of being normal is preserved by intersections with compact open subgroups.
\begin{lemma}\label{normalintersect}
Let $\Gamma\subseteq \PGL_2(F_{S})$ be a normal plectic subgroup.
For every subset $\mbf{S}_{\Gamma}\subseteq S_1\subseteq S$ and every compact open subgroup $U\subseteq \PGL_2(F_{S\setminus S_1})$ the group $\Gamma_U$ is a normal plectic subgroup.
\end{lemma}
\begin{proof}
The compact open subgroup $U$ is commensurable with the stabilizer of every vertex in the simplicial complex $\prod_{\p\in S\setminus S_1} \mathcal{T}_{\mathcal{L}_{\Gamma,\p}}$.
Thus, we may reduce to the case that $U$ is equal to such a stabilizer, for which the claim is obvious.
\end{proof}


\section{Plectic Jacobians}\label{Sec: Plectic Jacobians}
\subsection{Mumford varieties and zero cycles}
Let $(F_S,\iota)$ be an embedded local \'etale algebra and $\Gamma\subseteq \PGL_2(F_S)$ a plectic subgroup.
Since $\mathcal{L}_{\Gamma,\p}\subseteq \PP(\C)$ is a compact subset, its complement $\Omega_{\Gamma,\p}= \PP(\C)\setminus \mathcal{L}_{\Gamma,\p}$ and, thus, the product 
\[
\Omega_\Gamma := \prod_{\p \in S} \Omega_{\Gamma,\p}= \PP(\C_S)\setminus \mathcal{L}_\Gamma,
\] 
carry the structure of a rigid analytic variety over $\C$.
Under mild assumptions the quotient of $\Omega_\Gamma$ by the action of $\Gamma$ exists in the category of rigid analytic varieties.
Indeed, using methods as in Mumford's seminal paper \cite{Mumford} or as in \cite[Lemma 4]{Berkovich} one gets the following result:
\begin{theorem}\label{quotients}
Suppose that $\Gamma\subseteq \PGL_2(F_S)$ is a torsion-free plectic subgroup.
Then there exists a pair $(X_\Gamma, \pi_\Gamma)$, unique up to isomorphism, consisting of a smooth rigid analytic variety $X_\Gamma$ and a surjective \'etale morphism
\[
\pi_\Gamma\colon \Omega_\Gamma \too X_\Gamma
\]
such that
\begin{itemize}
\item[$\bfcdot$] $\pi_\Gamma \circ \gamma = \pi_\Gamma$ for all $\gamma\in\Gamma$ and,
\item[$\bfcdot$] if $x,y\in \Omega_\Gamma$ satisfy $\pi_\Gamma(x)=\pi_\Gamma(y)$, then $\gamma(x)=y$ for some $\gamma\in\Gamma$.
\end{itemize}
The rigid analytic variety $X_\Gamma$ is proper if and only if $\Gamma$ is normal.
\end{theorem}

\begin{remark}\label{choices}
One can realize a rigid analytic variety as a quotient as above in many different ways.
For starters, one can enlarge $F_S$ by choosing a finite extension of each $F_\p$ together with an embedding to $\C$ that extends $F_\p$.
One can also conjugate $\Gamma$ by elements of $\PGL_2(F_S)$ and switch factors in case $F_\p\cong F_\q$ for $\p,\q \in S$.
It is not clear whether there are additional ways to change $\Gamma$ (see \cite{Alon} for a discussion of this question in the case of cocompact subgroups).
\end{remark}
To keep track on the data involved in the uniformization we make the following definition.

\begin{definition}\label{def: Mum}
Let $(F_S,\iota)$ be an embedded local \'etale algebra.
A \emph{connected Mumford variety} (over $(F_S,\iota)$) is a triple $(\Gamma,X_\Gamma,\pi_\Gamma)$ where $\Gamma\subseteq \PGL_2(F_S)$ is a torsion-free normal plectic subgroup and $(X_\Gamma,\pi)$ arises from Theorem \ref{quotients}. In particular, $X_\Gamma$ is proper.
A \emph{Mumford variety} (over $(F_S,\iota)$) is a finite tuple of connected Mumford varieties (over $(F_S,\iota)$). 
\end{definition}
\begin{remark}
    We require the plectic groups involved in the definition of Mumford varieties to be normal so that Corollary \ref{fingen} holds. This will allow us to define plectic Jacobians in Definition \ref{Def: plectic jacobians}.  
\end{remark}
\noindent We often suppress $\Gamma$ and $\pi_\Gamma$ from the notation and simply call $X_\Gamma$ a (connected) Mumford variety.
Given a rigid analytic variety $X$ over $\C$ we write $\mathcal{Z}(X)$ for the group of $0$-cycles on $X$, that is, the free abelian group on the set $X(\C)$.
Let
\[
\deg\colon \mathcal{Z}(X)\too \Z,\quad \sum_{x\in X(\C)} n_x [x] \mapsto \sum_{x\in X(\C)} n_x
\]
denote the degree map, and $\mathcal{Z}_0(X)=\ker(\deg)$ the group of $0$-cycles of degree zero.
If $X_\Gamma$ is a connected Mumford variety, the projection $(\pi_\Gamma)_\ast\colon \mathcal{Z}(\Omega_\Gamma) \to \mathcal{Z}(X_\Gamma)$ induces an isomorphism
\begin{equation}\label{coinvariants}
\mathcal{Z}(\Omega_\Gamma)_\Gamma:=\mrm{H}_0\big(\Gamma,\mathcal{Z}(\Omega_\Gamma)\big)\xlongrightarrow{\sim} \mathcal{Z}(X_\Gamma).
\end{equation}
Moreover, there is a natural $\Gamma$-equivariant isomorphism
\[
\psi\colon \bigotimes_{\p\in S} \mathcal{Z}(\Omega_{\Gamma,\p}) \xlongrightarrow{\sim} \mathcal{Z}(\Omega_{\Gamma})
\]
sending elementary tensors $\otimes_{\p\in S} [x_\p]$ to $S$-tuples $[(x_\p)_{\p\in S}]$.
For any plectic subgroup $\Gamma$ define the $\Gamma$-submodule $\mathcal{Z}_{\plectic}(\Omega_{\Gamma})\subseteq \mathcal{Z}(\Omega_{\Gamma})$ of $0$-cycles of degree $(0,\ldots,0)$ as the image of the tensor product $\bigotimes_{\p\in S} \mathcal{Z}_0(\Omega_{\Gamma,\p})$ under $\psi$.
Thus, taking $\Gamma$-coinvariants and invoking \eqref{coinvariants} yields a homomorphism
\[
\Theta_\Gamma\colon \mathcal{Z}_{\plectic}(\Omega_{\Gamma})_\Gamma \too \mathcal{Z}(X_{\Gamma})
\]
for every Mumford variety $X_\Gamma$.

\begin{definition}
Let $X_\Gamma$ be a connected Mumford variety.
The group of $0$-cycles of degree $(0,\ldots,0)$ of $X_\Gamma$ is defined as
\[
\mathcal{Z}_{\plectic}(X_{\Gamma}):= \mrm{Im}(\Theta_\Gamma)\subseteq \mathcal{Z}(X_{\Gamma}).
\]
Given a Mumford variety $X=X_{\Gamma_1}\coprod\ldots\coprod X_{\Gamma_n}$ we put
\[
\mathcal{Z}_{\plectic}(X):=\bigoplus_{i=1}^{n}\mathcal{Z}_{\plectic}(X_{\Gamma_i}).
\]
 \end{definition}

\begin{remark}
The group $\mathcal{Z}(X_{\Gamma})$ clearly only depends on the rigid analytic variety $X_\Gamma$ itself, whereas the subgroup
$\mathcal{Z}_{\plectic}(X_{\Gamma})$ a priori depends on chosen uniformization of $X_\Gamma$.
Nevertheless, note that changing $\Gamma$ by conjugation or switching factors as explained in Remark \ref{choices} preserves the subgroup of $0$-cycles of degree $(0,\ldots,0)$.
Whenever $X_\Gamma$ is one-dimensional, it is clear that $\mathcal{Z}_{\plectic}(X_{\Gamma})$ recovers the group of divisors of degree $0$.
\end{remark}

\noindent Let $\Aut(\C/F_\p)$ be the group of continuous automorphisms of $\C$ fixing $\iota_\p(F_\p)$.
The product
\[
\Aut(\C/F_S)=\prod_{\p\in S} \Aut(\C/F_\p).
\]
naturally acts on the product $\PP(\C_S)$ of projective spaces and commutes with the action of $\PGL_2(F_S)$.
Thus, for every plectic subgroup $\Gamma\subset \PGL_2(F_S) $ the rigid analytic variety $\Omega_\Gamma$ is stable under $\Aut(\C/F_S)$-action.
Moreover, if $\Gamma$ is torsion-free, the action descends to an action on $X_\Gamma$ and thus on $\mathcal{Z}(X_\Gamma)$.
For every Mumford variety $X$ the subgroup $\mathcal{Z}_\plectic(X)\subseteq \mathcal{Z}(X)$ is a $\Aut(\C/F_S)$-submodule

\smallskip
\noindent Let $(F_{S'}',\iota')$ be an embedded local \'etale algebra of the same characteristic and residue characteristic as $(F_S,\iota)$.
Given Mumford varieties $X$ and $X'$ over $(F_S,\iota)$ and  $(F'_{S'},\iota')$  respectively, the product $X\times X'$ has the structure of a Mumford variety over $(F_S\times F_{S'}', \iota\times\iota')$.
Moreover, restricting the isomorphism
\begin{align*}
\mathcal{Z}(X)\otimes \mathcal{Z}(X')&\too \mathcal{Z}(X\times X')\\
x\otimes x' &\mapstoo (x,x')
\end{align*}
to $0$-cycles of degree $(0,\ldots,0)$ yields an isomorphism
\begin{align}\label{monoidal1}
\mathcal{Z}_{\plectic}(X)\otimes \mathcal{Z}_{\plectic}(X')\too \mathcal{Z}_{\plectic}(X\times X')
\end{align}
of $\Aut(\C/F_S\times F_{S'}')$-modules.

\subsection{Integration theory}
Let $\mathcal{L}$ be a compact, totally disconnected space and $A$ an abelian group.
We write $\mathscr{M}(\mathcal{L},A)$ for the space of $A$-valued measures on $\mathcal{L}$, i.e.,
\[
\mathscr{M}(\mathcal{L},A):=\Hom_{\Z}(C(\mathcal{L},\Z),A).
\]
In order to have a well-behaved integration theory, we require $A$ to be a finite free $\Z$-module. Then, if $M$ is a second-countable prodiscrete abelian group, there exists a canonical integration paring
\begin{align}\label{integration1}
\mathscr{M}(\mathcal{L},A)\otimes C(\mathcal{L},M)\too A\otimes_{\Z} M
\end{align}
constructed as follows:
first, note that for any discrete abelian group $N$ the canonical map
\[
C(\mathcal{L},\Z)\otimes N \xlongrightarrow{\sim} C(\mathcal{L},N)
\]
is an isomorphism and, thus, there exists a canonical integration pairing
\[
\mathscr{M}(\mathcal{L},A)\otimes C(\mathcal{L},N)\too A\otimes N
\]
that is functorial in $N$.
For $M$ second-countable prodiscrete, choose a basis of neighbourhoods $\left\{U_i\right\}_i$ of the identity of $M$ consisting of open subgroups with $U_{i+1}\subseteq U_i$.
Since  the canonical map $C(\mathcal{L},M)\xrightarrow{\sim} \varprojlim_{i}C(\mathcal{L},M/U_i)$ is an isomorphism, and $A$ is a finite free $\Z$-module, we obtain the sought after pairing by considering
\[
\mathscr{M}(\mathcal{L},A)\otimes \varprojlim_{i}C(\mathcal{L},M/U_i)\too \varprojlim_{i} A\otimes M/U_i=M.
\]
\begin{remark}
If $\mathcal{L}$ can be written as a product
\[
\mathcal{L}=\mathcal{L}_1\times \mathcal{L}_2
\]
of profinite spaces.
Then the map
\begin{align}
\begin{split}\label{products}
C(\mathcal{L}_1,\Z) \otimes_\Z C(\mathcal{L}_2,\Z) &\too C(\mathcal{L}_1\times \mathcal{L}_2,\Z)\\
f_1\otimes f_2 &\mapstoo [(x_1,x_2)\mapsto f_1(x_1)\cdot f_2(x_2))]
\end{split}
\end{align}
is an isomorphism of abelian groups.
\end{remark}
\smallskip
\noindent Now, let us turn to the situation of a plectic subgroup $\Gamma\subseteq \PGL_2(F_S)$ for some embedded local \'etale algebra $F_S$.
The group $\Gamma$ acts on $\mathcal{L}_{\Gamma,S}$ and, thus, on $\scr{M}(\mathcal{L}_{\Gamma,S},A)$.
By \eqref{products} the Steinberg module $\St_{\Gamma,S}$ is a quotient of the function space
\[
C(\mathcal{L}_{\Gamma,S},\Z)\cong \bigotimes_{\p\in S} C(\mathcal{L}_{\Gamma,\p},\Z).
\]
Thus, for any abelian group $A$ the group
\[
\scr{M}_{\plectic}(\mathcal{L}_{\Gamma,S},A):=\Hom_\Z(\St_{\Gamma,S},A)
\]
is a $\Gamma$-submodule of the space of $A$-valued measures on $\mathcal{L}_{\Gamma,S}$.
More precisely, it is the space of measures vanishing on functions which are constant in at least one variable.
For any second-countable prodiscrete abelian group $M$,
the $\Gamma$-module $\St_{\Gamma,S}^{\mrm{ct}}(M)$ is defined to be the quotient of $C(\mathcal{L}_{\Gamma,S},M)$ by the submodule generated by those functions  which are constant with respect to at least one variable. 
If $A$ is a finite free $\Z$-module,
the integration pairing \eqref{integration1} descends to a pairing
\begin{align}\label{integration2}
\langle\cdot, \cdot\rangle\colon \scr{M}_{\plectic}(\mathcal{L}_{\Gamma,S},A) \times \St_{\Gamma,S}^{\mrm{ct}}(M) \too A\otimes_\Z M.
\end{align}

\subsection{Universal measures and multiplicative integrals}
Let us fix a plectic subgroup $\Gamma\subseteq \PGL_2(F_S)$ and consider the maximal torsion-free quotient $H_{\Gamma,S}$ of the group of $\Gamma$-coinvariants $\left(\St_{\Gamma,S}\right)_\Gamma$ of $\St_{\Gamma,S}$.
By definition, there is a canonical identification
\begin{equation}\label{univmeasure}
\scr{M}_{\plectic}(\mathcal{L}_{\Gamma,S},A)^{\Gamma}=\Hom_\Z(H_{\Gamma,S},A)
\end{equation}
for every torsion-free abelian group $A$.
\begin{definition}
The \emph{universal $\Gamma$-invariant measure} on $\mathcal{L}_{\Gamma,S}$ is the measure  
\[
\mu_\Gamma^{S}\in \scr{M}_{\plectic}(\mathcal{L}_{\Gamma,S},H_{\Gamma,S})^\Gamma
\]
corresponding under \eqref{univmeasure} to the identity on $H_{\Gamma,S}$.
\end{definition}

The next lemma is a direct consequence of the definitions.
\begin{lemma}
Let $S_1,S_2\subseteq S$ be disjoint subsets and $\Gamma_{i}\subseteq \PGL_2(F_{S_i})$, $i=1,2$, plectic subgroups.
There is a natural identification
\begin{align}\label{productdecomp}
H_{\Gamma_1\times\Gamma_2,S_1\cup S_2}=H_{\Gamma_1,S_1}\otimes H_{\Gamma_2,S_2}
\end{align}
of abelian groups.
Moreover, the universal measure of $\Gamma_1\times \Gamma_2$ naturally factors as
\begin{align}\label{fubini}
\mu_{\Gamma_1\times\Gamma_2}^{S_1\cup S_2}= \mu_{\Gamma_1}^{S_1} \otimes \mu_{\Gamma_2}^{S_2}.
\end{align}
\end{lemma}

\smallskip
Given second-countable prodiscrete subgroups $M_{j}$, $i=1,\ldots,k$, we consider the topology on $M_1 \otimes_\Z \ldots\otimes_\Z M_k$ generated by opens of the form
\[
M_{1}\otimes_\Z\ldots\otimes_\Z M_{\ell-1}\otimes_\Z U_\ell \otimes_\Z M_{\ell+1} \otimes \ldots \otimes_\Z M_k
\]
with $1\leq \ell \leq k$ and $U_\ell\subseteq M_{\ell}$ an open subgroup.
Its completion $M_1 \widehat{\otimes} \ldots\widehat{\otimes} M_k$ is a second-countable prodiscrete group.
We are mostly interested in the case \[
\C^\times_{S,\otimes}:=\widehat{\otimes}_{\p\in S}\C^\times.
\]
As the componentwise action of $\Aut(\C/F_S)$ on $\otimes_{\p\in S}\C^\times$ is continuous, it induces an action on the completion $\C_{S,\otimes}^\times.$
\begin{lemma}
	The map
	\begin{align*}
	\Phi_{S} \colon \mathcal{Z}_{\plectic}(\Omega_{\Gamma})&\too\St_{\Gamma,S}^{\mrm{ct}}(\C_{S,\otimes}^\times),\\
	\otimes_{\p\in S}([x_\p]-[y_\p])&\mapstoo\left[(t_\p)_{\p\in S}\mapsto\bigotimes_{\p\in S}\left(\frac{t_\p-x_\p}{t_\p-y_\p}\right)\right]
	\end{align*}
	is a well-defined $\Gamma\times\Aut(\C/F_{S})$-equivariant homomorphism.
\end{lemma}
\begin{proof}
Note that the function 
	\[
	\left(\frac{(\cdot)-x_\p}{(\cdot)-y_\p}\right)\colon\mathcal{L}_{\Gamma,\p}\too\C^\times
	\]
	is continuous for every $\p\in S$.
	Thus $\Phi_{S}$ is well-defined.
	The equivariance with respect to the $\Aut(\C/F_{S})$-action is obvious, while the equivariance with respect to the $\Gamma$-action holds because we are considering the continuous $S$-Steinberg module.
\end{proof}

Let $A$ be a finite free $\Z$-module and $\mu\in \scr{M}_{\plectic}(\mathcal{L}_{\Gamma,S},A)$ a measure.
For a $0$-cycle $D\in \mathcal{Z}_{\plectic}(\Omega_{\Gamma})$ of degree $(0,\ldots,0)$ the multiplicative integral of $D$ with respect to $\mu$ is defined to be
\[
\mint_{D}\omega_\mu:=\langle \mu, \Phi_S(D) \rangle\ \in\ A\otimes_\Z \C_{S,\otimes}^{\times},
\]
where $\langle\cdot,\cdot\rangle$ is the integration pairing \eqref{integration2}.
The multiplicative integral is $\Aut(\C/F_S)$-equivariant, where $\Aut(\C/F_S)$ acts on $H_{\Gamma,S}\otimes_\Z \C_{S,\otimes}^{\times}$ through the second factor.

\begin{lemma}\label{dense}
Let $A$ be a finite free $\Z$-module and $\mu\in \scr{M}_{\plectic}(\mathcal{L}_{\Gamma,S},A)$ a surjective measure.
The image of the multiplicative integral
\[
\mint_{(-)}\omega_\mu\colon \mathcal{Z}_{\plectic}(\Omega_{\Gamma}) \too  A\otimes_\Z \C_{S,\otimes}^{\times}
\]
is dense.
\end{lemma}
\begin{proof}
By hypothesis, integration over $\mu$ induces a surjective homomorphism from $\St_{\Gamma,S}^{\mrm{ct}}(M)$ to $M$ for  every second-countable prodiscrete group $M$.
Therefore, it suffices to prove that the image of the homomorphism $\Phi_{\p}$ is dense  for all $\p \in S$.
By changing coordinates on $\PP(\C)$ we may assume that $\infty \notin \mathcal{L}_{\Gamma,\p}$.
In this new coordinates, the image of $\Phi_{\p}$ contains all polynomials with $\C$-coefficients, which do not vanish on $\mathcal{L}_{\Gamma,\p}$.
Hence, density follows from the $p$-adic Stone--Weierstraß theorem.
\end{proof}

The preceding lemma can be applied to the following situation:
assume that $\Gamma$ is normal, then the abelian group $H_{\Gamma,S}$ is finitely generated and free by Corollary \ref{fingen}.
Thus, the universal multiplicative integral
\[
\mint_{D}\omega_{\mu_\Gamma^{S}}\ \in\ H_{\Gamma,S}\otimes_\Z \C_{S,\otimes}^{\times}
\]
exists for all $D\in \mathcal{Z}_{\plectic}(\Omega_{\Gamma})$ and $\mu_\Gamma^{S}$ is surjective by construction.

\subsection{Plectic Jacobians and the Abel--Jacobi map}
Let $(F_S,\iota)$ be an embedded local \'etale algebra and $X_\Gamma$ a connected Mumford variety over it.
Since the universal measure $\mu_\Gamma^{S}$ is $\Gamma$-invariant, it defines a $\Aut(\C/F_S)$-equivariant homomorphism
\begin{align}\label{almostAJ}
\mint_{(-)} \omega_{\mu_\Gamma^{S}}\colon \mathcal{Z}_{\plectic}(\Omega_{\Gamma})_\Gamma \too  H_{\Gamma,S}\otimes_\Z \C_{S,\otimes}^{\times}.
\end{align}
In order to obtain a homomorphism from the group of $0$-cycles of degree $(0,\ldots,0)$ on $X_\Gamma$, we mod out by the relations coming from the kernel of the quotient map $\Theta_\Gamma\colon\mathcal{Z}_{\plectic}(\Omega_{\Gamma})_\Gamma\to\mathcal{Z}_{\plectic}(X_{\Gamma})$.
Thus, we define 
\[
\Lambda_\Gamma:= \mint_{\ker(\Theta_\Gamma)} \omega_{\mu_\Gamma^{S}}\ \subseteq\ H_{\Gamma,S}\otimes_\Z \C_{S,\otimes}^{\times} .
\]

\begin{definition}\label{Def: plectic jacobians}
The \emph{plectic Jacobian} of a connected Mumford variety $X_\Gamma$ is the quotient
	\[
	\mathcal{J}_{\plectic}(X_\Gamma):=\left(H_{\Gamma,S}\otimes_\Z \C_{S,\otimes}^{\times}\right)/\overline{\Lambda_\Gamma},
	\]
	where $\overline{\Lambda_\Gamma}$ denotes the closure of $\Lambda_\Gamma$.
The plectic Abel--Jacobi map of $X_\Gamma$ is the homomorphism
\[
\alpha(X_\Gamma)\colon \mathcal{Z}_{\plectic}(X_{\Gamma})\too \mathcal{J}_{\plectic}(X_\Gamma)
\]
induced by \eqref{almostAJ}.
Given a Mumford variety $X=X_{\Gamma_1}\coprod\ldots\coprod X_{\Gamma_n}$ we put
\begin{align*}
\mathcal{J}_{\plectic}(X)&:=\mathcal{J}_{\plectic}(X_{\Gamma_1})\oplus\cdots\oplus\mathcal{J}_{\plectic}(X_{\Gamma_n})\\
\intertext{and}
\alpha(X)&:=\alpha(X_{\Gamma_1}) \oplus\cdots\oplus \alpha(X_{\Gamma_n}).
\end{align*}
\end{definition}

The action of $\Aut(\C/F_S)$ on $H_{\Gamma,S}\otimes_\Z \C_{S,\otimes}^{\times}$ descends to an action on $\mathcal{J}_{\plectic}(X_\Gamma)$.
From the construction and Lemma \ref{dense} we deduce:
\begin{proposition}
Let $X$ be a Mumford variety over $(F_S,\iota)$.
Its plectic Jacobian $\mathcal{J}_{\plectic}(X)$ is a prodiscrete group with a continuous $\Aut(\C/F_S)$-action.
The plectic Abel--Jacobi map $\alpha(X)$ is $\Aut(\C/F_S)$-equivariant and has dense image.
\end{proposition}

\begin{remark}
When $X_\Gamma$ is a one-dimensional, $\Lambda_\Gamma=\overline{\Lambda_\Gamma}$ is a discrete subgroup.
Moreover, the plectic Jacobian $\mathcal{J}_{\plectic}(X_\Gamma)$ of $X_\Gamma$ can be identified with the set of $\C$-valued points of the Jacobian of $X_\Gamma$ via the Manin--Drinfeld's $p$-adic uniformization theorem (see for example \cite[Section 2.5]{Das05}), and the plectic Abel--Jacobi map coincides with the $p$-adic uniformization of the classical one.
In particular, the kernel of the plectic Abel--Jacobi map is just the subgroup of principal divisors.
Is it possible to give a general geometric description of the kernel of the plectic Abel--Jacobi map?
\end{remark}

\begin{example}
Let $X_{\Gamma}$ be a connected Mumford variety over $(F_S,\iota)$ such that $\mbf{S}_{\Gamma}\neq S$, then $\St_{\Gamma,S}=0$ and consequently $\mathcal{J}_\plectic(X_\infty)=0$. This generalizes the fact that the Jacobian of the projective line is trivial.
\end{example}

For the reminder of this section we fix two embedded local \'etale algebras $(F_S,\iota)$ and $(F_{S'}',\iota')$  of the same characteristic and residue characteristic.
Given Mumford varieties $X$ and $X'$ over $(F_S,\iota)$ and $(F_{S'}',\iota')$ we have already seen in \eqref{monoidal1} that the group of $0$-cycles of degree $(0,\ldots,0)$ on their product equals the tensor product of $\mathcal{Z}_{\plectic}(X)$ and $\mathcal{Z}_{\plectic}(X')$.
In the following we will show that the same is true for their plectic Jacobians.
First, assume that $X=X_{\Gamma}$ and $X'=X_{\Gamma'}$ are connected.
The decomposition \eqref{productdecomp} induces an isomorphism
\[
\Pi\colon \left(H_{\Gamma,S}\otimes_\Z \C_{S,\otimes}^{\times}\right) \widehat{\otimes} \left(H_{\Gamma',S'}\otimes_\Z \C_{S',\otimes}^{\times}\right)
\xlongrightarrow{\sim}
\left(H_{\Gamma\times \Gamma',S'\coprod S'}\right)\otimes_\Z \C_{S\coprod S',\otimes}^{\times}.
\]
\begin{theorem}\label{monoidalthm}
Let $X_{\Gamma}$ and $X_{\Gamma'}$ be connected Mumford varieties over $(F_S,\iota)$ and $(F'_{S'},\iota')$.
Then $\Pi$ descends to an isomorphism
\begin{align}\label{productdecomp2}
\mathcal{J}_{\plectic}(X_\Gamma)\hspace{1mm}\widehat{\otimes}\hspace{1mm}\mathcal{J}_{\plectic}(X_{\Gamma'})\xlongrightarrow{\sim} \mathcal{J}_{\plectic}(X_{\Gamma\times \Gamma'})
\end{align}
of topological $\Aut(\C/F_S\times F_{S'}')$-modules.
\end{theorem}
\begin{proof}
By the Künneth formula for group homology we observe that
\[
\ker(\Theta_{\Gamma\times\Gamma'})=\ker(\Theta_\Gamma)\otimes_\Z \mathcal{Z}_{\plectic}(\Omega_{\Gamma'})_{\Gamma'}\ \oplus\ \mathcal{Z}_{\plectic}(\Omega_{\Gamma})_{\Gamma}\otimes_\Z\ker(\Theta_{\Gamma'}).
\]
Hence, the product decomposition \eqref{fubini} of the universal measure in conjunction with Lemma \ref{dense} implies that
\[
\overline{\Lambda_{\Gamma\times\Gamma'}}=
\overline{\Lambda_\Gamma}\ \widehat{\otimes}\ \left(H_{\Gamma',S'}\otimes_\Z \C_{S',\otimes}^{\times}\right)
\ \oplus\ 
\left(H_{\Gamma,S} \otimes_\Z \C_{S,\otimes}^{\times}\right)\ \widehat{\otimes}\ \overline{\Lambda_{\Gamma'}},
\]
which proves the claim.
\end{proof}

\begin{corollary}\label{monoidalcor}
Let $X$ and $X'$ be connected Mumford varieties over $(F_S,\iota)$ and $(F'_{S'},\iota')$.
The isomorphisms \eqref{productdecomp2} induce an isomorphism
\begin{align}\label{monoidal2}
\mathcal{J}_{\plectic}(X)\hspace{1mm}\widehat{\otimes}\hspace{1mm}\mathcal{J}_{\plectic}(X')\xlongrightarrow{\sim} \mathcal{J}_{\plectic}(X \times X')
\end{align}
of topological $\Aut(\C/F_S\times F'_{S'})$-modules.
Moreover, the Abel--Jacobi map 
\[
\alpha(X\times X') = \alpha(X) \otimes \alpha(X')
\]
is compatible with the decomposition \eqref{monoidal1} and \eqref{monoidal2}.
\end{corollary}

\subsection{Hecke correspondences}\label{Hecke}
We fix an embedded local \'etale algebra $(F_S,\iota)$.
\begin{definition}
A morphism from a connected Mumford variety $X_\Gamma$ over $(F_S,\iota)$ to another connected Mumford variety $X_{\Gamma'}$ over $(F_S,\iota)$ is a rigid analytic analytic map of the form
\[
f(g)\colon X_\Gamma \too X_{\Gamma'},\quad [x] \mapsto [gx],
\]
where $g\in\PGL_2(F_S)$ is an element such that $g\Gamma g^{-1}\subseteq \Gamma'$ is a finite index subgroup.

\noindent A morphism between two Mumford varieties $X=X_{\Gamma_1}\coprod\ldots\coprod X_{\Gamma_n}$ and $X'=X_{\Gamma'_1}\coprod\ldots\coprod X_{\Gamma'_m}$ over $(F_S,\iota)$ is a rigid analytic map such that there an injection $\sigma=\sigma(f)\colon (1,\ldots,n)\into (1,\ldots,m)$ and a morphism of connected Mumford varieties $f_i\colon X_{\Gamma_i}\to X_{\Gamma'_{\sigma(i)}}$ for each $1\leq i \leq n$.
\end{definition}
The composition of two morphism of Mumford varieties is clearly a morphism of Mumford varieties.
Note that a morphism $f$ of Mumford varieties is surjective if and only if $\sigma(f)$ is surjective.
Let $f\colon X\to X'$ be a morphism of Mumford varieties.
In case $X$ and $X'$ are connected, the pushforward map
\[
f_\ast\colon \mathcal{Z}(X) \too \mathcal{Z}(X')
\]
on $0$-cycles corresponds to the canonical map between the groups of coinvariants under the identification \eqref{coinvariants}.
Similarly, the pullback map
\[
f^\ast\colon \mathcal{Z}(X') \too \mathcal{Z}(X)
\]
is given by the transfer map on coinvariants.
In particular, both of them restrict to $\Aut(\C/F_S)$-equivariant homomorphisms
\[
f_\ast\colon \mathcal{Z}_{\plectic}(X) \too \mathcal{Z}_{\plectic}(X'),\qquad
f^\ast\colon \mathcal{Z}_{\plectic}(X') \too \mathcal{Z}_{\plectic}(X).
\]
on groups of $0$-cycles of degree $(0,\ldots,0)$.
Let $g$ be an element of $\PGL_2(F_S)$ and $\Gamma,\Gamma'\subseteq \PGL_2(F_S)$ plectic subgroup such that $g\Gamma g^{-1}\subseteq \Gamma'$ is a finite index subgroup.
It follows that  $\St_{g\Gamma g^{-1},S}=\St_{\Gamma',S}$.
Functoriality of the operation of taking coinvariants and the existence of transfer maps for finite index subgroups gives homomorphisms
\[
f_\ast\colon H_{\Gamma,S} \too H_{\Gamma',S},\qquad
f^\ast\colon H_{\Gamma',S} \too H_{\Gamma,S},
\]
that in turn induce $\Aut(\C/F_S)$-equivariant pushforward and pullback homomorphisms between plectic Jacobians:
\[
f_\ast\colon \mathcal{J}_{\plectic}(X) \too \mathcal{J}_{\plectic}(X'),\qquad
f^\ast\colon \mathcal{J}_{\plectic}(X') \too \mathcal{J}_{\plectic}(X).
\]
Going through the construction one easily verifies the following:
\begin{lemma}\label{functorial}
Let $f\colon X \to X'$ be a morphism of Mumford varieties over $(F_S,\iota)$.
The Abel--Jacobi map commutes with pushforward and pullback, that is,
\begin{align*}
\alpha(X') \circ f_\ast = f_\ast\circ\alpha(X)\quad \mbox{and}\quad \alpha(X) \circ f^\ast = f^\ast\circ\alpha(X').
\end{align*}
\end{lemma}

\begin{proposition}\label{fibprod}
Let $f_1\colon X_1\to X$ and $F_2\colon X_2\to X$ morphisms of Mumford varieties over $(F_S,\iota)$.
The fibre product of rigid analytic varieties $X_1 \times_X X_2$ is a Mumford variety over $(F_S,\iota)$.
Moreover, the canonical projection maps $X_1 \times_X X_2 \to X_i$, $i=1,2$, are morphisms of Mumford varieties over $(F_S,\iota)$.
\end{proposition}
\begin{proof}
It is enough to consider the case where $X=X_\Gamma$, $X_1=X_{\Gamma_1}$ and $X_2=X_{\Gamma_2}$ are connected, that $\Gamma_1$ and $\Gamma_2$ are finite index subgroup of $\Gamma$, and that the morphisms $f_1$ and $f_2$ correspond to the inclusion maps.
For $g\in \Gamma$ consider the morphism
\[
X_{\Gamma_1\cap g^{-1}\Gamma_2 g} \too X_{\Gamma_1} \times X_{\Gamma_2},\quad [x] \mapsto ([x], g[x]).
\]
Their disjoint union induces a closed embedding
\[
\coprod_{g\in \Gamma_2\backslash \Gamma /\Gamma_1}
X_{\Gamma_1\cap g^{-1}\Gamma_2 g} \intoo X_{\Gamma_1}\times X_{\Gamma_2}
\]
whose image is the fibre product of $X_{\Gamma_1}$ and $X_{\Gamma_2}$ over $X_\Gamma$.
This proves the first claim.
The second claim also follows directly from this description of the fibre product.
\end{proof}

\begin{definition}
Let $X$ and $X'$ be two Mumford varieties over $(F_S,\iota)$.
A Hecke correspondence from $X$ to $X'$ consists of a Mumford variety $\widetilde{X}$ over $(F_S,\iota)$ and together with a surjective morphism $p\colon \widetilde{X} \to X$ and a morphism $q\colon\widetilde{X'}\to X$.
Given two Hecke correspondences $(\widetilde{X},p,q)\colon X\to X'$ and $(\widetilde{X'},p',q')\colon X'\to X''$  we define their composition via
\[
\big(\widetilde{X'},p',q'\big)\circ \big(\widetilde{X},p,q\big):=\big(\widetilde{X}\times_{X'}\widetilde{X'}, p\circ\pi_{\widetilde{X}},q'\circ \pi_{\widetilde{X'}}\big),
\]
where $\pi_{\widetilde{X}}\colon \widetilde{X}\times_{X'}\widetilde{X'} \to \widetilde{X}$ and $\pi_{\widetilde{X'}}\colon \widetilde{X}\times_{X'}\widetilde{X'} \to \widetilde{X'}$ denote the projection morphisms.
By Proposition \ref{fibprod} the composition of Hecke correspondences is well-defined.
\end{definition}

Write $\mrm{Mum}_{/(F_S,\iota)}$ for the category of Mumford varieties over $(F_S,\iota)$ with Hecke correspondences as morphism.
Let $\Aut(\C/F_S)$-$\mrm{Mod}$ (resp.~$\Aut(\C/F_S)$-$\mrm{Mod}^{\mrm{ct}}$) be the category of $\Aut(\C/F_S)$-modules (resp.~continuous $\Aut(\C/F_S)$-modules).
Finally, denote by
\[
\nu\colon \Aut(\C/F_S)\mbox{-}\mrm{Mod}^{\mrm{ct}}\longrightarrow \Aut(\C/F_S)\mbox{-}\mrm{Mod}
\]
the forgetful functor.
The following theorem summarizes what we have proven so far:
\begin{theorem}
Let $(F_S,\iota)$ be an embedded local \'etale algebra.
Assigning the group of $0$-cycles of degree $(0,\ldots,0)$ (resp.~the plectic Jacobian) to a Mumford variety over $(F_S,\iota)$ define the functors
\begin{align*}
\mathcal{Z}_{\plectic}\colon \mrm{Mum}_{/(F_S,\iota)} &\too \Aut(\C/F_S)\mbox{-}\mrm{Mod},\\
\mathcal{J}_{\plectic}\colon \mrm{Mum}_{/(F_S,\iota)} &\too \Aut(\C/F_S)\mbox{-}\mrm{Mod}^{\mrm{ct}}.
\end{align*}
Moreover, the plectic Abel--Jacobi map is natural transformation
\[
\mrm{AJ}\colon \mathcal{Z}_{\plectic} \Longrightarrow \nu\circ\mathcal{J}_{\plectic}.
\]
\end{theorem}
 
\begin{remark}
Instead of working with Mumford varieties over a fixed embedded local \'etale algebra one could work with the category of all Mumford varieties over $\C$.
In this setting, Corollary \ref{monoidalcor} could be reformulated as saying that the Abel--Jacobi map is a natural transformation between symmetric monoidal functors.
\end{remark}


\section{Modularity}\label{Sec: Modularity}
Let $\overline{\Q}\subseteq\bb{C}$ denote the algebraic closure of $\Q$ contained in the field of complex numbers.
For  a totally real field $F$ of degree $d$ we consider the set of its embeddings $\Sigma_\infty:=\Hom_\Q(F,\overline{\Q})$ into $\overline{\Q}$ and the set $S_p$ of its $p$-adic places.
Let $\C_p$ denote the completion of a fixed algebraic closure of $\Q_p$.
The choice of an embedding $\iota_p\colon\overline{\Q}\hookrightarrow\C_p$ determines a map
\[
\theta\colon\Sigma_\infty\too S_p 
\]
sending $\sigma\in \Hom_\Q(F,\overline{\Q})$ to the $p$-adic place $\p_\sigma$ induced by the composition $\iota_p\circ\sigma$.
Let $\Sigma\subseteq \Sigma_\infty$ be a non-empty subset such that  $S:=\theta(\Sigma)$ has the same cardinality of $\Sigma$.
The fixed embedding $\iota_p$ induces an embedding on the local \'etale algebra
\[
F_S:= \prod_{\p \in S} F_{\p}.
\]
Set $p_S:=\prod_{\p\in S} \p$. Consider a square-free ideal $\n^{-}$ coprime to $p_S$, and write $\omega(\n^{-})$ for its number of prime factors.
Suppose that $\omega(\n^{-})\equiv d\pmod{2}$,
then there exists a quaternion algebra $B/F$ ramified exactly at the places in
\[
S\cup \left\{\mathfrak{q}\mid\n^{-}\right\}\cup\left(\Sigma_\infty\setminus\Sigma\right).
\]
Let $\mathcal{H}^\pm=\PP(\bb{C})\setminus\PP(\R)$ denote the union of the upper and lower complex half planes. For every $\sigma\in\Sigma$ we choose an isomorphism $B_\sigma:=B\otimes_{F,\sigma}\R\cong \mrm{M}_2(\R)$ so that the
diagonal embedding of $B$ into $\prod_{\sigma\in\Sigma}B_\sigma$ induces an action of $B^\times$ on $\mathcal{H}_\Sigma^\pm:=\prod_{\sigma\in\Sigma} \mathcal{H}^\pm$. Let $E_{\Gal}$ denote the Galois closure of $\sigma(F)$ in $\overline{\Q}$ for some $\sigma\in \Sigma$, and  consider
\[
G_{\Sigma}:=\big\{\tau\in \Gal(E_\mrm{Gal}/\Q)\ \big\vert\ \tau\circ\Sigma=\Sigma\big\},
\qquad
E:=E_{\Gal}^{G_{\Sigma}}.
\]

\subsection{Quaternionic Shimura varieties }\label{Hilbert}
Denote by $\bb{A}_f$ the ring of finite adeles over $\Q$ and $K\subseteq B^\times(\bb{A}_f)$ be a compact open subgroup such that $\mrm{nr}(K)=\widehat{\mathcal{O}}_F^\times$.
The quaternionic Shimura variety
\[
X_K(\C)=B^\times\backslash \left(\mathcal{H}_\Sigma^\pm\times B^\times(\bb{A}_f)\right) /K
\]
has a canonical model over the reflex field $E$. 
In the following we only consider level structures of the form $K=\prod_{\q}K_\q$, where $\q$ runs through the finite places of $F$, such that $K_\p$ is the unique maximal compact subgroup of $B_\p$ for all $\p\in S$.
We are mostly interested in level structures of type $K_0(\n^+)$ where $\n^+\subseteq \mathcal{O}_F$ is an ideal coprime to $p_S\n^-$ and
\begin{itemize}
\item[$\bfcdot$] $K=\prod K_\q$ with $K_\q$ maximal compact for $q\nmid\n^+$,
\item[$\bfcdot$] $K_\q$ corresponds to the group of integral matrices that are congruent to an upper triangular matrix mod $\n^+$ under an isomorphism $B_\q\cong \mrm{M}_2(K_\q)$ for every $\q\mid \n^+$.
\end{itemize}

By definition, the reflex field is a subfield of $\overline{\Q}$.
Thus, the embedding $\iota_p$ induces a $p$-adic place $v$ on $E$.
The $E_v$-rigid analytic space associated to $X_K$ admits a $p$-adic uniformization  by \cite[Section 5]{Varshavsky}.
Concretely, let $D/F$ be the totally definite quaternion algebra obtained from $B/F$ by switching invariants at $S$ and $\Sigma$, and choose an isomorphism $D_\p\cong \mrm{M}_2(F_\p)$ for every $\p\in S$. Diagonally embedding $D$ into $\prod_{\p\in S} D_\p$ induces an action of $D^\times$ on $\Omega_S:=\prod_{\p\in S}\Omega_\p$, where $\Omega_\p=\PP(\C_p)\setminus\PP(F_\p)$ denotes the Drinfeld's upper half plane associated to $F_\p$.
Moreover, we identify $D(\bb{A}^S_f)$ and $B(\bb{A}^S_f)$ where $\bb{A}^{S}_F$ denotes the ring of finite adeles away from $S$.
This induces an action of $D^\times$ on $B(\bb{A}^S_f)$.
Moreover, $D^\times$ acts on $B(F_\p)/K_\p\cong \Z$ via translation by the $p$-adic valuation of the determinant.
The $p$-adic uniformization theorem then states that there is an isomorphism of rigid analytic varieties
\[
X_{K(\n)}(\C_p)\cong D^\times\backslash\big(\Omega_S\times B(\bb{A}_f)/K\big).
\]
Furthermore, the group of connected components is given by the narrow class group 
$\mrm{Cl}^+_F$ of $F$.
Let $\{g_i\}_i\subset B(\bb{A}^S)$ be elements whose reduced norms form a system of representatives for $\mrm{Cl}^+_F$.
The images of the groups $\Gamma_i=g_i\cdot K\cdot g_i^{-1}\cap D^\times$ in $\PGL_2(F_S)$ are cocompact subgroups and thus plectic.
We may write
\begin{align}\label{classgroup}
X_K(\C_p)\cong \coprod_{i=1}^{h_F^+}X_{\Gamma_i}.
\end{align}
In particular, $X_K$ carries the structure of a Mumford variety over $(F_S,\iota)$.

\subsection{Geometric modularity}
Let $A_{/F}$ be an elliptic curve with multiplicative reduction at all $\p\in S$.
Let $A_\p$ denote the analytification of $A_{F_\p}$ seen as a rigid analytic variety over $\C_p$.
By Tate's uniformization theorem there exists a unique element $q_\p \in F_\p^\times$ with positive $p$-adic valuation such that
\[
A_\p \cong \C_p^\times /q_\p^\Z.
\]
In particular, $A_\p$ is a Mumford curve and, thus, the product
\[
A_S:=\prod_{\p\in S} A_\p
\]
is a Mumford variety over $(F_S,\iota_p)$.
More generally, given a Tate period $\widetilde{q}_\p\in F_\p^\times$ for every $\p\in S$ we denote by $A_{\widetilde{q}_\p}$ the corresponding elliptic curve over $\C_p$.
By Theorem \ref{monoidalthm}, the plectic Jacobian of the product $A_{\widetilde{q}_S}:=\prod_{\p \in S} A_{\widetilde{q}_p}$ is simply the completed tensor product of the groups $A_\p(\C_p)$ .
In particular, there exists a canonical surjective homomorphism
\[
\phi_{\widetilde{q}_S}\colon \C_{S,\otimes}^\times\too \mathcal{J}_\plectic(A_{\widetilde{q}_S}).
\]

\begin{theorem}\label{modtheorem}
Let $A_{/F}$ be a modular elliptic curve of conductor $N=p_S\cdot\n^{+}\cdot\n^{-}$ for some ideal $\n^{+}\subseteq \mathcal{O}_F$ coprime to $p_S\cdot\n^{-}$.
There is a surjective $\Aut(\C_p/F_S)$-equivariant homomorphism
\[
\varphi_A\colon\mathcal{J}_{\plectic}(X_{K(\n^{+})})\too \mathcal{J}_{\plectic}(A_S).
\]
\end{theorem}
\begin{proof}
This is mostly just a reformulation of the results of \cite[Section 4.5]{PlecticHeegner}.
We sketch the main points.
First, let $\Gamma_i\subseteq\PGL_2(F_S)$ be the subgroups defined in \eqref{classgroup}.
The Hecke algebra acts on the direct sum $\oplus_{i}H_{\Gamma_i,S}$.
The modularity of $A_{/F}$ implies that there is a unique $\Z$-free quotient $H_A$ of $\oplus_{i}H_{\Gamma_i,S}$ on which the Hecke algebra acts via the character associated to $A$.
Moreover, multiplicity one implies that $H_A$ has rank one.
The choice of a generator of $H_A$ then yields a homomorphism
\begin{align}\label{quotient}
\bigoplus_{i=1}^{h_F^+} \big(H_{\Gamma_i,S}\otimes_\Z \C^\times_{S,\otimes}\big)\too \C^\times_{S,\otimes}
\end{align}
that in turn induces the homomorphism
\[
\psi_A\colon \bigoplus_{i=1}^{h_F^+}\mathcal{Z}_{\plectic}(\Omega_{\Gamma_i})_{\Gamma_i} \too \C^\times_{S,\otimes}
\]
through $p$-adic integration. Now, \cite[Theorem 4.10]{PlecticHeegner} implies that there exist Tate periods $\widetilde{q}_\p\in F_\p$ for every $\p\in S$ with the following properties:
\begin{enumerate}[(a)]
\item[$\bfcdot$] the subgroups $q_\p^{\Z}$ and $\widetilde{q}_\p^{\Z}$ are commensurable, and
\item[$\bfcdot$] the homomorphism
\[
\phi_{\widetilde{q}_A}\circ\psi_A\colon\bigoplus_{i=1}^{h_F^+} \mathcal{Z}_{\plectic}(\Omega_{\Gamma_i})_{\Gamma_i}\too  \mathcal{J}_{\plectic}(A_{\widetilde{q}_S})
\]
factors through $\mathcal{Z}_{\plectic}(X_{K(\n^{+})})$ and hence through $\mathcal{J}_{\plectic}(X_{K(\n^{+})})$.
\end{enumerate}
By choosing an isogeny $A_{\widetilde{q}_\p}\to A_\p$ for each $\p\in S$, we obtain the claimed homomorphism
\[
\mathcal{J}_{\plectic}(X_{K(\n^{+})})\too \mathcal{J}_{\plectic}(A_{S}).
\]
\end{proof}

\begin{remark}
The Hecke module $\oplus_{i} H_{\Gamma_i,S}$ is isomorphic to the space of cusp forms on $X_{K(\n^{+})}$.
Although the results of \cite{PlecticHeegner} were only stated in the case of eigenforms corresponding to elliptic curves, one can easily adapt the methods to arbitrary eigenforms.
In particular, Theorem \ref{modtheorem} can be reformulated as follows:
the plectic Jacobian of $X_{K(\n^{+})}$ embeds into the plectic Jacobian of a product of Shimura curves $\{X_\p\}_{\p\in S}$ with each $X_\p$ admitting a $p$-adic uniformization by a $\p$-arithmetic subgroup of the definite quaternion algebra $D$.
More precisely, the level structure used to define $X_\p$ is $K_0(p_S \p^{-1}n^+)$.
This observation suggests the following question:
let $X_\Gamma$ be a connected Mumford variety over an embedded local \'etale algebra $(F_S,\iota)$.
Do Mumford curves $\{X_\p\}_{\p\in S}$ exists such that
\begin{itemize}
\item[$\bfcdot$] the connected components of each $X_\p$ are connected Mumford curve associated to groups of the form $\Gamma_U$ for an
open compact subgroup $U\subseteq \PGL_2(F_{S\setminus\{\p\}})$, 
\item[$\bfcdot$] and there exists an $\Aut(\C/F_S)$-equivariant embedding
\[
\mathcal{J}_{\plectic}(X_\Gamma) \intoo \mathcal{J}_{\plectic}(\prod_{\p \in S} X_\p)?
\]
\end{itemize}
If such a collection of Mumford curves existed, $X_{\Gamma}$ could be seen as a braided version of the product $\prod_{\p\in S} X_\p$ -- at least, on the level of plectic Jacobians -- reflecting the fact that $\Gamma$ can be thought of as a braided version of the product $\prod_{\p\in S}\Gamma_\p$.
\end{remark}


\bibliography{Jacobians}
\bibliographystyle{alpha}

\end{document}